\documentclass[1p]{elsarticle_modified}

\journal{CGP16037}

\usepackage{amsfonts}
\usepackage{amssymb}
\usepackage{amsmath}
\usepackage{amsthm}
\usepackage{amsbsy}

\usepackage{float}
\usepackage{kotex}

\usepackage{paralist}

\usepackage{caption,subcaption}

\usepackage[colorlinks]{hyperref}
\usepackage{color}

\usepackage{graphicx}
\usepackage{ulem}
\usepackage[scr=dutchcal]{mathalfa}

\newcommand{\stkout}[1]{\ifmmode\text{\sout{\ensuremath{#1}}}\else\sout{#1}\fi}

\newtheorem{thm}{Theorem}[section]
\newtheorem{cor}[thm]{Corollary}
\newtheorem{prop}[thm]{Proposition}
\newtheorem{lem}[thm]{Lemma}
\newtheorem{conj}[thm]{Conjecture}

\theoremstyle{definition}

\theoremstyle{remark}
\newtheorem{rem}[thm]{Remark}

\theoremstyle{proof}

\def\Pcal{{\mathcal{P}}}

\DeclareMathOperator{\st}{\mathscr{st}}

\begin{document}

\begin{frontmatter}

\title{ A classification of polyhedral graph \\ by combinatorially rigid vertices}

\author{Yunhi Cho} 
\address{Department of Mathematics, University of Seoul, Seoul 02504 , Korea}
\ead{yhcho@uos.ac.kr}

\author{Seonhwa Kim\fnref{s_kim}}
\address{Center for Geometry and Physics, Institute for Basic Science, Pohang 37673, Korea}
\fntext[s_kim]{The second author was supported by IBS-R003-D1.}
\ead{ryeona17@ibs.re.kr}

\begin{abstract}
 When the number of  non-triangular faces adjacent to a vertex $v$ is less than or equal to three,  the vertex  $v$ will be called (\emph{combinatorially}) \emph{rigid}. 
 We study the number of rigid vertices and suggest  a conjecture on  a classification of polyhedra. 
\end{abstract}


\end{frontmatter}


\tableofcontents

\section{Introduction}\label{intro}

A \emph{polyhedral graph} $P$ is a planar graph given by the 1-skeleton of a strictly-convex Euclidean polyhedron. Equivalently, a polyhedral graph is a 3-connected planar graph with no loops and multiple edges by Steinitz's theorem (for a reference, see \cite{ziegler_lectures_1995}). In this article, a strictly-convex Euclidean realization of a polyhedral graph is  called simply by a \emph{polyhedron}. We assume that the ambient space containing polyhedral graphs is a $2$-sphere and we say that two polyhedral graphs are the \emph{same} if there is a plane isotopy between them. 

The non-triangular degree  of a vertex $v$ is the number of  non-triangular faces adjacent to $v$. 
A vertex $v$ of $P$ is \emph{combinatorially rigid} if the non-triangular degree of $v$ is less than or equal to 3. 
Note that this definition is given  purely combinatorially but the property is related to a  rigidity phenomenon of spherical vertex figures of a geometric 3-dimensional polyhedron (See Section \ref{sec:geometricrigid}). 
For the sake of convenience, we will omit the term `\emph{combinatorially}' in the article unless necessary. 

An important and direct consequence from the definition of a rigid vertex  is Lemma \ref{existencerigid}: there exists a rigid vertex for any polyhedral graph. The same combinatorial idea   was also used   in Lemma 18, \cite{montcouquiol_deformations_2012} and the similar statement however might have previously appeared  although the authors couldn't find an earlier reference. 

We would like think about the notion of a rigid vertex as follows.
If $P$ is \emph{simple} or \emph{simplicial} which means that all vertices are 3-valent or  all faces are 3-gonal respectively, then all vertices in $P$ are obviously rigid. A question naturally aries : how many non-rigid vertices can be in a polyhedral graph?
In summary, our study shows that if we restrict the number of rigid vertices, it is also restrictive to find distinct polyhedra and so we ask a question in Conjecture \ref{finiteconj} whether there are only finite \emph{irreducible} polyhedra under the fixed number of rigid vertices. We expect that the study on rigid vertices may help to understand polyhedral graphs as how far it is from both extremes: simple and simplicial.


Here is the first main result of classifying polyhedral graphs by the number of rigid vertices. 

\begin{thm}\label{mainclassification} 
	Let  $\Pcal_k$ be the set of polyhedral graphs with  $k$  rigid vertices. Then  we get the following  classification.

	\begin{enumerate}
		\item   For $k \leq 3$, $\Pcal _ k$ is the empty set.
		\item $  \Pcal _ 4  $ has  only one element, the  tetrahedron.
		\item $ \Pcal _ 5 $ has only two elements, the 4-gonal pyramid and the 3-gonal bipyramid.
		\item For a positive integer $k \geq 6$, $\Pcal_k$ is   infinite.
	\end{enumerate}
	
\end{thm}

\begin{figure}[h]
	{
		$$
		\begin{array}{c|c|c|c}
		
		\Pcal_k \text{ for } k\leq3 & \Pcal_4 & \Pcal_5 & \Pcal_k \text{ for } k\geq 6 \\
		\hline
		\varnothing ~~&~~	\vcenter{\hbox{\includegraphics[scale=0.6]{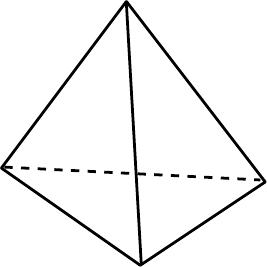}} } ~~~~&~~~
		\vcenter{\hbox{\includegraphics[scale=0.6]{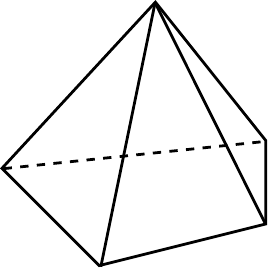}}} ~~,~~~ \vcenter{\;
			\hbox{\includegraphics[scale=0.6]{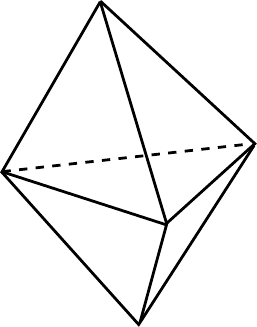}}}  ~~~&~~~
		
		infinite 
		\end{array}
		$$
	}	
	\caption{There are only three polyhedra if the number of  rigid vertices is less than 6.}
\end{figure}

We obtain two additional results as follows, which will be discussed in Section \ref{sec:add}. 
\begin{thm}\label{ninevertex}
	If a polyhedron   has  fewer than 9 vertices, then all vertices are rigid.
\end{thm}
\begin{thm}\label{symmetryrigid}
	A Euclidean convex polyhedron with regular faces  has only rigid vertices.
\end{thm}

In the proof of Theorem \ref{mainclassification}, we glue  two polyhedra together along a face in order  to obtain an infinite family of polyhedra with fewer than a fixed number of rigid vertices. 
We predict that this is an essential way to produce an infinite number of polyhedra with the number of rigid vertices fixed. Let us see the statement more precisely. 

A polyhedral graph $P$ is \emph{reducible}\footnote{It might be better to use the term \emph{decomposable}. But the term \emph{decomposable} is already used in several areas. Typically, it has been commonly used in the sense of Minkowski sums. Even in rigidity context like 
	\cite{ConnellySchlenker}, it is also used for a special kind of non-convex polyhedra which can be decomposed into tetrahedra without adding vertices. Therefore we will use the term \emph{reducible}.}  if $P$ is the 1-skeleton graph of a polyhedron obtained from two polyhedra glued together along a face and  $P$ is \emph{irreducible} otherwise.
From the  point of view of planar graphs, $P$ is \emph{reducible} means that  $P$ is decomposed  into  $P_1$ and $P_2$ along a {separating edge-path cycle}. When we consider an edge-path cycle $C$ on $P$, there are two connected regions  separated by $C$. For $P_1$ and $P_2$ to be  polyhedral graphs, there is an adjacent edge toward each region at each vertex in $C$. Hence all vertices in separating cycle should be at least  $4$-valent.

Note that  irreducibility is combinatorial. More precisely, if we consider a separating cycle in a realization, it may be coplanar or not, i.e.  $P = P_1 \cup P_2$ and  the intersection $P_1 \cap P_2$ is a planar $n$-gon  or a convex hull of the separating cycle which has a positive volume as in Figure \ref{fig:separingcycle}.

\begin{figure}[H]
	$$\vcenter{\hbox{\includegraphics[scale=0.6]{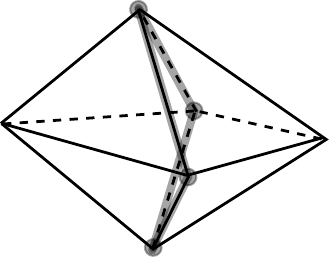}  }}P
	~=~~
	\vcenter{\hbox{\includegraphics[scale=0.6]{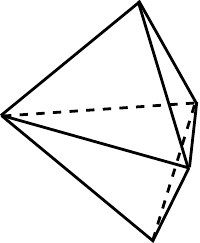}  }}P_1 \cup~~
	\vcenter{\hbox{\includegraphics[scale=0.6]{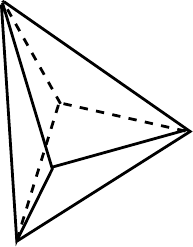}  }}P_2 ,~~~ P_1\cap P_2  =
	\vcenter{\hbox{\includegraphics[scale=0.6]{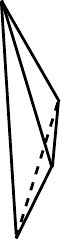}  }}$$
	\caption{A reducible polyhedron $P$ is decomposed into $P_1$ and $P_2$, where the seperating cycle is not coplanar.}
	\label{fig:separingcycle}
\end{figure}
For a reducible polyhedron,
we can increase the number of non-rigid vertices to an unlimited extent under the restriction of  the number of  rigid vertices, as  attaching an intermediate $n$-gonal prisms along the separating cycle, as in Figure \ref{fig:iteratedprism}.
But  it is not easy to  produce different polyhedral graphs of fixed number of rigid vertices  if we restrict to only irreducible polyhedra. We checked this prediction through further classifications as in Remark \ref{rmk:p6}. 
Let  $\Pcal_k^\text{irr}$ be the set of all irreducible polyhedral graph with $k$ rigid vertices. We conjecture the following.
\begin{conj}\label{finiteconj}
	Each $\Pcal_k^\text{irr}$ is finite. 
\end{conj}  

Finally, we present the following classification table by Theorem \ref{mainclassification} and Remark \ref{rmk:p6}.

\begin{figure}[H]
	{
		$$
		\begin{array}{c|c|c|c|c}
		
		\Pcal_4^\text{irr} & \Pcal_5^\text{irr} & \Pcal_6^\text{irr} & \Pcal_7^\text{irr} &\cdots \\
		\hline
		\vcenter{\;\hbox{\includegraphics[scale=0.6]{tetrahedron.pdf}} } &
		\vcenter{\;\hbox{\includegraphics[scale=0.6]{4pyramid.pdf}}}
		&
		\vcenter{\hbox{\includegraphics[scale=0.6]{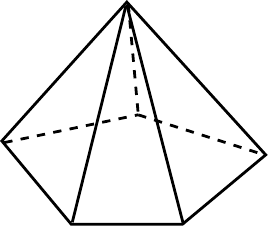}}}
		,~ \vcenter{\hbox{\includegraphics[scale=0.6]{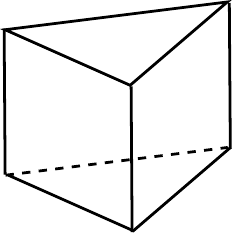}  }}
		,~
		\vcenter{\hbox{\includegraphics[scale=0.6]{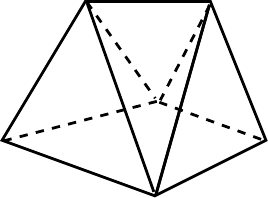}  }}
		,~
		\vcenter{\hbox{\includegraphics[scale=0.6]{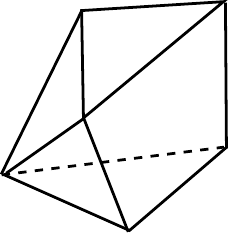}  }}
		&
		\vcenter{ \hbox{finite?} \; \; \; \hbox{~~or} \; \; \hbox{infinite?}}
		&
		?
		\end{array}
		$$
	}	
	\caption{$|\Pcal_{\leq 3}^\text{irr}| = 0, |\Pcal_4^\text{irr}| = 1, |\Pcal_5^\text{irr}| = 1 \text{ and } |\Pcal_6^\text{irr}| = 4$.}
\end{figure}

\section{Combinatorially rigid vertex and spherical figure at a vertex} \label{sec:geometricrigid}
In this section, we briefly review the relation between combinatorially rigid vertices and geometric realizations of  polyhedral graph. 
Many parts are from the authors' other paper \cite{CK2016}. 
Let $V$, $E$ and $F$ be the number of vertices, edges and faces respectively. Let $V_k$ or $F_k$ the  the number of $k$-valent vertices or $k$-gonal faces respectively.
Let us begin  with reviewing a well-known lemma (for example, see p.237 \cite{grunbaum_convex_2003}). 

\begin{lem}
	\label{lem:v3f3} 
	For every polyhedral graph $P$, we have
	$$ V_3 + F_3 = \sum_{n\geq5} (n-4)(V_n + F_n) + 8.$$
\end{lem} 
\begin{proof}
	Each edge is adjacent to  vertices and faces exactly twice respectively. Hence we get  
	\begin{equation}\label{eq:2ev3f3}
	2E=\sum_{n\geq3}n V_n =\sum_{n\geq3}n F_n.
	\end{equation}
	
	
	Recall  Euler's polyhedron formula $V - E + F =2$ and the following completes the proof.
	\begin{align*}
	V_3+F_3 &=   4E - 4V - 4F + V_3 + F_3 +8 \\
	&=   \sum_{n\geq3}n V_n + \sum_{n\geq3}n F_n -4\sum_{n\geq3}V_n  - 4\sum_{n\geq3}F_n +V_3 +F_3 +8 \\
	&= \sum_{n\geq5} (n-4)V_n + \sum_{n\geq5} (n-4)F_n + 8
	\end{align*}
\end{proof}

The following existence result is an easy consequence of the above lemma;  it is a starting point of the whole story.
\begin{lem}\label{existencerigid}
	For every  polyhedral graph $P$, there always exists a rigid vertex.
\end{lem}
\begin{proof}
	Suppose there is no rigid vertex.
	It is obvious that $ V_3 =0 $ and  all vertices  meet  at most $\deg (v) -4 $ triangle faces. Therefore we get
	$$
	0<3F_3 \leq \sum_{n\geq 5} (n-4) V_n
	$$
	This inequality contradicts  the following inequality obtained by Lemma \ref{lem:v3f3}.
	$$ F_3 \geq \sum_{n\geq 5} (n-4)V_n + 8 $$ 
\end{proof}
The following lemma shows that a combinatorially rigid vertex is actually a rigid neighborhood in a geometric 3-space like Euclidean, hyperbolic or spherical 3-space. 
\begin{lem}
	\label{rigidvertexthm}
	Let $P$ and $Q$ be two polyhedra of the same polyhedral graph where corresponding dihedral angles coincide. If a vertex $v$ is combinatorially rigid, then there is an ambient isometry $\phi$  which transforms between sufficiently small neighborhoods at $v$, i.e. $\phi: B_\epsilon(v(P)) \rightarrow B_\epsilon(v(Q))$.
\end{lem}
As an application of rigid vertices, we can prove the following rigidity theorem using an essentially different method from Stoker's proof.
\begin{thm}[J. Stoker \cite{stoker_geometrical_1968} ]\label{EDrigidity} 
	Let P be a polyhedral graph. 
	If its edge lengths and dihedral angles are given, then all strictly-convex realizations are isometric to each other. 
\end{thm}
We omit the proofs for Lemma \ref{rigidvertexthm} and Theorem \ref{EDrigidity} because it would digress from the subject. Moreover
we remark that combinatorially rigid vertices can be enhanced  to deal with non-convex cases, called \emph{strong-rigid} vertex. See \cite{CK2016} for the proofs and the other details.

\section{A classification of polyhedral graphs by the number of rigid vertices}

This section is devoted to  the proof of Theorem \ref{mainclassification}. For the sake of convenience, let us introduce  new notation  $V^{\text{rig}}$ and $V^{\text{non}}$ which denote the number  of rigid vertices and non-rigid vertices respectively. For example, $V_3 = V_3^{\text{rig}}$ and $V_4= V_4^{\text{rig}} +V_4^{\text{non}} $.
We say a vertex $v$ is  \emph{totally triangular} if all adjacent faces are triangles.
Let us begin with the following lemma.

\begin{lem}\label{lem:1stIneqVrig} For every polyhedral graph, the following inequality holds.
	$$8-V^{\text{rig}}\leq F_3 \leq 2 V^{\text{rig}}-4.$$
\end{lem}
\begin{proof} The left inequality is  obtained by $V^{\text{rig}}+ F_3\geq V_3+F_3 \geq 8$.
	When we count the maximal  number of triangles at a vertex $v$ of $\deg(v)=n$, there are at most $n$ or $(n-4)$ triangles   for a rigid or nonrigid vertex  respectively. Therefore the total maximal number  is $ 3V_3 +4V_4^{\text{rig}}+5V_5^{\text{rig}}+\cdots +V_5^{\text{non}}+2V_6^{\text{non}}+3V_7^{\text{non}}+\cdots.$
	Since the counting is triply redundant  because each triangle has three corners, so we have
	\begin{align*}
	3F_3 &\leq 3V_3 +4V_4^{\text{rig}}+5V_5^{\text{rig}}+\cdots +V_5^{\text{non}}+2V_6^{\text{non}}+3V_7^{\text{non}}+\cdots\\
	&=4(V_3+V_4^{\text{rig}}+V_5^{\text{rig}}+\cdots)+(V_5+2V_6+3V_7+\cdots) -V_3\\
	&=4V^{\text{rig}}-V_3+\sum_{n\geq 5}{(n-4) V_n}.
	\end{align*}
	
	As applying the inequality  $F_3+V_3 \geq \sum_{n\geq 5}{(n-4) V_n} + 8 + 4V^{\text{rig}}-4V^{\text{rig}} $ from Lemma \ref{lem:v3f3}, we get
	$$F_3 \geq 4V^{\text{rig}}-V_3 +\sum_{n\geq 5}{(n-4) V_n}+8-4V^{\text{rig}}\geq 3F_3+8-4V^{\text{rig}}$$
	and
	it proves the right inequality.
\end{proof}
The above Lemma  gives the following theorem immediately.
\begin{thm}\label{thm:leastVrig}
	Every polyhedral graph has at least four rigid vertices. Therefore  $\Pcal _ k$ is the empty set for $k = 0,1,2,3$.
\end{thm}
\begin{proof}
	$8-V^{\text{rig}}\leq  2 V^{\text{rig}}-4$ in Lemma \ref{lem:1stIneqVrig} implies $ V^{\text{rig}} \geq 4$.
\end{proof}
We introduce a criterion using the number of rigid vertices so as  to determine whether a polyhedral graph is a tetrahedron or not. The following lemma plays a  crucial role in the classification.

\begin{lem}\label{lem:2ndIneqVrig} For a polyhedral graph $P$,
	the following two inequalities hold simultaneously if and only if  the polyhedron $P$ is a tetrahedron.
	\begin{enumerate}[\indent(a)]
		\item $-2V_3+4V^{\text{rig}}+\sum_{n\geq 5}{(n-4) V_n}<3F_3$
		\item $V_4^{\text{rig}}+V_5^{\text{rig}}+V_{\geq 6}<3$.
	\end{enumerate}
	
\end{lem}
\begin{proof}
	If $P$ is a tetrahedron, the two inequalities are satisfied trivially. Let us prove the converse.
	The first inequality (a) implies that a totally triangular 3-valent vertex exists in $P$. If there is no totally triangular 3-valent vertex, then the 3-valent vertices $V_3$ have at most two adjacent triangles. So we have
	$$\aligned 3F_3 &\leq 2V_3 +4V_4^{\text{rig}}+5V_5^{\text{rig}}+\cdots +V_5^{\text{non}}+2V_6^{\text{non}}+3V_7^{\text{non}}+\cdots\\
	&=4V^{\text{rig}}+\sum_{n\geq 5}{(n-4) V_n}-2V_3.
	\endaligned$$
	
	\begin{figure}[h]
		\begin{center}
			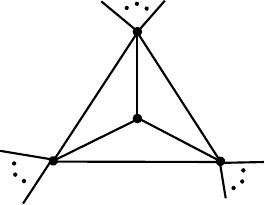
			\caption{Neighborhood of a totally triangular vertex $v$}
			\label{fig:totallytriangularvertices}
		\end{center}
	\end{figure}
	
	This contradicts  (a). Now, let us look at a totally triangular 3-valent vertex $v$ and the neighboring vertices $x$, $y$ and $z$ as in Figure \ref{fig:totallytriangularvertices}. Each vertex of them meets at least two triangle faces and hence contribute to $V_4^\text{rig}$, $V_5^\text{rig}$ or $V_{\geq6}$.
	If  inequality (b) holds, at least one of $x$,$y$ and $z$  should be 3-valent.  However,   if one or two vertices of $x$, $y$ and $z$  are 3-valent,  the planar graph of $P$ cannot be 3-connected. Therefore, it  contradicts  Steinitz's theorem  unless $x$,$y$ and $z$ all are 3-valent. This implies $P$ itself is a tetrahedron.
\end{proof}

\begin{thm}\label{thm:P4} If a polyhedral graph $P$ has four rigid vertices, then $P$ is a tetrahedron.
\end{thm}
\begin{proof}
	
	Let us assume that $V^{\text{rig}} =4$. Then we get $F_3=4$ by Lemma \ref{lem:1stIneqVrig}.
	The following inequalities from Lemma \ref{lem:v3f3},
	$$4+4  \geq  V_3 + F_3= \sum_{n\geq 5}{(n-4) V_n}+\sum_{n\geq 5}{(n-4) F_n} +8 \geq 8$$
	should be  equalities.
	We have $V_5=V_6=V_7=\cdots=0$ and $F_5=F_6=F_7=\cdots=0$, so  $V_5^{\text{rig}}=0$. Moreover we get $F_3=V_3=4$ and hence  $V_4^{\text{rig}}=V^\text{rig}-V_3^\text{rig}=0$.
	
	Therefore, the two conditions of Lemma \ref{lem:2ndIneqVrig} are satisfied and it implies that the polyhedron $P$ is a tetrahedron.
\end{proof}

Before analyzing the cases for $\Pcal_5$, we need  some preparation. For  $P \in \Pcal_5$, we have the following propositions.

\begin{prop}\label{prop:p5_v5f5zero}
	For a polyhedral graph $P$ with five rigid vertices, we have  $$V_{\geq 5} = F_{\geq5} = 0. $$
\end{prop}
\begin{proof}
	If there is a vertex or a face of degree greater than 4, then
	\begin{equation}\label{eq:prop_five}
	\sum_{n\geq 5}{(n-4) V_n}+\sum_{n\geq 5}{(n-4) F_n} >0
	\end{equation}
	and we will see this induces a contradiction.
	At first, by the above inequality (\ref{eq:prop_five}), we get
	$$V_3 \leq V^\text{rig} < 6 \leq 2\sum_{n\geq 5}{(n-4) V_n}+3\sum_{n\geq 5}{(n-4) F_n} +4.$$
	It implies the following by adding $-3V_3 + \sum_{n\geq 5}{(n-4) V_n} +4 V^\text{rig}$ to both sides,
	$$
	-2V_3 + 4 V^\text{rig} + \sum_{n\geq 5}{(n-4) V_n} < 3 \sum_{n\geq 5}{(n-4) V_n} + 3\sum_{n\geq 5}{(n-4) F_n} + 24 -3V_3=3F_3.$$
	Hence we get that condition (a) of Lemma \ref{lem:2ndIneqVrig} always holds.
	Secondly, by  $V^\text{rig}=5$ and Lemma \ref{lem:1stIneqVrig}, we know
	\begin{equation}\label{eq:prop:five3}
	3 \leq F_3 \leq 6 ~\text{ and }~ 2 \leq V_3 \leq 5.
	\end{equation}
	This implies that
	\begin{equation}\label{eq:prop_five2}
	\sum_{n\geq 5}{(n-4) V_n}+\sum_{n\geq 5}{(n-4) F_n} \leq 3.
	\end{equation}
	Also from the equality (\ref{eq:2ev3f3}) in the proof of Lemma \ref{lem:v3f3}, we have
	\begin{equation}\label{mod2}V_3+V_5+V_7+\cdots\equiv F_3+F_5+F_7+\cdots\equiv 0 \pmod 2\end{equation}
	
	If $V_3=2$, then $F_3=6$ by Lemma \ref{lem:v3f3} but this contradicts  the inequality (\ref{eq:prop_five}). If $V_3=3$, then $F_3=6$ and $V_5=1$ by  formula (\ref{mod2}) and $V_{\geq 6}=0$ by Lemma \ref{lem:v3f3}. Hence we know $V_4^\text{rig} + V_5^{rig} =2$ and   condition (b) of Lemma \ref{lem:2ndIneqVrig} holds.
	If $V_3\geq 4$, then $V_4^\text{rig} + V_5^\text{rig} \leq 1$ and $V_{\geq6}\leq 1$ by the inequality (\ref{eq:prop_five2}). This also satisfies  condition (b).  Therefore   assumption (\ref{eq:prop_five}) implies the two conditions of Lemma \ref{lem:2ndIneqVrig} simultaneously, hence $P$ should be a tetrahedron but this contradicts the assumption of five rigid vertices.
\end{proof}

\begin{prop}
	There are only two possibility for $P$ on $\Pcal_5$ as follows.
	\begin{enumerate}[\indent(i)]
		\item $\left( \begin{aligned}
		&V_3=2, &&V_4=n+3, &&V_{\geq5}=0\\
		&F_3=6, &&F_4=n, &&F_{\geq5}=0
		\end{aligned}\right)$ for some $n\geq 0$
		\item $\left( \begin{aligned}
		&V_3=4, &&V_4=n+1, &&V_{\geq5}=0\\
		&F_3=4, &&F_4=n+1, &&F_{\geq5}=0
		\end{aligned}\right)$ for some $n\geq 0$
	\end{enumerate}
	where $n$ is the number of non-rigid vertices.
\end{prop}
\begin{proof}
	From  formula (\ref{mod2}), we can derive that $V_3$ and $F_3$ should be even numbers.
	By inequalities (\ref{eq:prop:five3}) and
	$V_3+F_3 =8$ from Proposition \ref{prop:p5_v5f5zero}, the only  possibility is $( V_3=2,~ F_3=6)$ or $(V_3=4,~ F_3=4)$, and then we can compute the relation between $V_4$ and $F_4$ using the equality (\ref{eq:2ev3f3}).
\end{proof}
Now we can clarify the elements of $\Pcal_5$.

\begin{thm}\label{thm:p5} There are only two polyhedral graphs with 5 rigid vertices. One is a 4-pyramid and the other is a 3-bipyramid.
\end{thm}
\begin{proof}
	At first, let	us check the case of  $V_3=2$,  $F_3=6$.
	All triangles are adjacent  to  only rigid vertices since	a non-rigid 4-valent vertex should not be adjacent to any triangle, i.e.
	$$3F_3 \leq 3 V_3 + 4 V_4^\text{rig} = 3\cdot 2 +4 \cdot  3 =  18 = 3F_3.$$
	The above inequality is an equality and
	it means that  all rigid vertices cannot meet any  4-gonal face. So if there is a 4-gonal face, it meets only 4-gonal faces. Since $P$ is connected, there is no 4-gonal face and the only possibility is the 3-gonal bipyramid.
	
	Secondly, for the case of  $V_3=4$, $F_3=4$ the second condition of Lemma \ref{lem:2ndIneqVrig} holds as follows, $$V_4^{\text{rig}}+V_5^{\text{rig}}+V_{\geq 6}=1+0+0<3.$$
	Therefore there is no totally triangular 3-valent vertex in this case and  each 3-valent vertex has at most  2 adjacent triangles. Similarly,
	$$3F_3 \leq 2 V_3 + 4 V_4^\text{rig} = 2\cdot 4 +4 \cdot  1 =  12 = 3F_3,$$
	and hence every 3-valent vertex meets exactly one 4-gonal face and two triangles. There are a single rigid 4-valent vertex which meets only triangles and $n$ non-rigid 4-valent vertices that meet only 4-gonal faces, so non-rigid vertices are not connected to rigid vertices. The connectedness of $P$ implies $n=0$, and hence there are only four 3-vertices, one rigid 4-valent vertex, four triangles and one 4-gonal face. Therefore $P$ must be a 4-gonal pyramid.
\end{proof}

If there are  more then 5 rigid vertices, it is easy to construct infinitely many polyhedra as follows.

\begin{thm}
	For $k\geq 6$, $\Pcal_k$ has infinitely many  combinatorial types.
\end{thm}
\begin{proof}
	For $\Pcal_6$, there is a  sequence of polyhedra with  exactly 6 rigid vertices, called  iterated 3-gonal prisms.
	\begin{figure}[h!]
		$$
		\vcenter{\;\;\hbox{\includegraphics[scale=0.7]{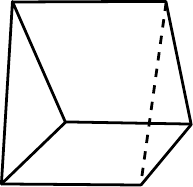}} } ~,~~
		\vcenter{\;\;\hbox{\includegraphics[scale=0.7]{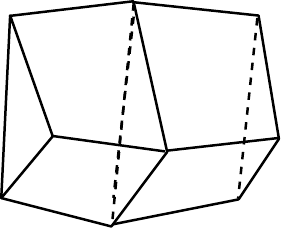}}} ~,~~
		\vcenter{\;
			\hbox{\includegraphics[scale=0.7]{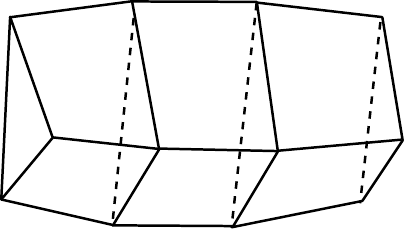}}}~,~\cdots\cdots  .
		$$
		\caption{Iterated 3-gonal prisms}
		\label{fig:iteratedprism}
	\end{figure}	
	If $k \geq 7$, we can consider a local move  which decomposes a  triangle  to three triangles and makes a new totally triangular 3-valent vertex,
	$$\vcenter{\hbox{\includegraphics[scale=0.6]{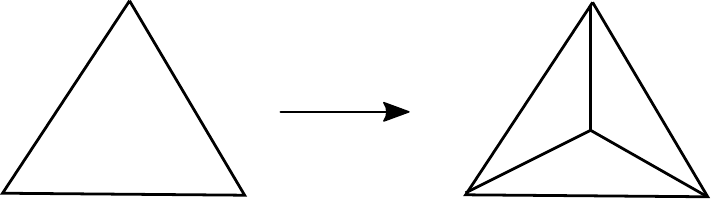}  }}$$	
	If this move is applied to a triangle adjacent to only rigid vertices, the number of rigid vertices increases by only one. Therefore, we can make an infinite family in $\Pcal_k$ for each $k\geq7$ from the iterated prisms in $\Pcal_6$.
\end{proof}

\begin{rem}\label{rmk:p6}
	In fact,  the authors also checked all combinatorial types in $\Pcal_6$.
	There are only a finite number (exactly six polyhedra  with 6 vertices) of combinatorial types except the above iterated 3-gonal prisms.
	Since there are only 7 polyhedra with 6 vertices (see \cite{steve_dutch_enumeration} for the classification of polyhedral graphs by the number of faces, which is dual to our cases),
	$$\Pcal_6 = \{\text{ irreducible polyhedral graph with 6 vertices}\}  \cup \{\text{ iterated 3-gonal prisms}\}.$$
	The proof of this classification of $\Pcal_6$ is similar to Theorem \ref{thm:p5}, but it is much more complicated and tedious. We don't present the proof in this article.
\end{rem}

\section{Additional results on rigid vertices}\label{sec:add}
\subsection{A lower bound}
From the  results in the previous section, one may observe that every polyhedron in $\Pcal_k$ for $k=4,5,6$ does not have any non-rigid vertex except iterated 3-gonal prisms. We can get an exact lower bound of the numbers of vertices if it has a non-rigid vertex. 
Theorem \ref{ninevertex} is an immediate consequence as a contraposition.
\begin{thm}\label{lowerbound}
	Let $\Pcal^{non}$ be the set of polyhedral graphs with a non-rigid vertex. Then, 
	$$ \min \{V(P) \mid P \in \Pcal^{non} \} =9 ,$$ where $V(P)$ is the number of vertices in $P$.
\end{thm}

\begin{proof}
	Let us define \emph{$i$-star} of $v$, denoted by $\st^i(v)$, as the union of $i$-cells  adjacent to $v$.  Let us  consider a Euclidean strictly-convex realization for a given polyhedral graph $P$.
	For a non-rigid vertex $v$, let us consider  $\st^2(v)$. All vertices in $\st^2(v)$ are distinct because $P$ can be realized as an intersection of half spaces in Euclidean space. Since $v$ is non-rigid, there are at most $\deg (v) - 4$ triangles.  Since $V(\st^2(v))$ should have at least 10  vertices unless $v$ is 4-valent and adjacent to only 4-gonal faces, we have $ \min \{V(P) \mid P \in \Pcal^{non} \} \geq 9 $. We already know  an example of $V(P) = 9$ which is the doubly iterated prism in $\Pcal_6$.
\end{proof}

\subsection{Regular faced polyhedra and rigid vertices}
Sometimes, we may consider a certain kind of special polyhedra with  symmetry or  transitivity.  If we check  the enumeration lists of such  polyhedra, we may recognize non-rigid vertices are very rare.
For example, let us consider  a Euclidean strictly-convex polyhedron with all regular faces. There is a complete classification: 5 Platonic solids, 13 Archimedean solids, an infinite number of prisms and anti-prisms and 92 Johnson-Zalgaller solids(see \cite{johnson_convex_1966}).
One can check the list one by one in order to prove Theorem \ref{symmetryrigid}, but we can prove it easily as follows.
\begin{proof}[Proof of Theorem \ref{symmetryrigid}]
	The facial angle of a $n$-gonal regular face is $\frac{ (n-2) \pi} {n}$.  Suppose that there is  a non-rigid vertex $v$, then $v$ is adjacent to at least four non-triangular faces.
	$$\text{Total angle sum at $v$} \geq \sum_{f \ni v}^{\deg{f}\geq 4} \frac{ (\deg{f}-2) \pi} {\deg{f}} \geq 4 \cdot \frac{\pi}{2}=2\pi$$
	This contradicts  convexity.
\end{proof}
Since all polyhedral graphs which are  vertex- or edge- transitive have regular faced  Euclidean realizations(see \cite{fleischner_transitive_1979}), we  get the following corollary.
\begin{cor}
	If a polyhedral graph is vertex- or edge- transitive, then all vertices are rigid.
\end{cor}



\begin{thebibliography}{}
	%
	
	\bibitem[CK]{CK2016} Yunhi~Cho, Seonhwa~Kim, {\em On the global rigidity of polyhedra by edge lengths and dihedral angels}, preprint, (2016)
	
	\bibitem[CS]{ConnellySchlenker} Robert~Connelly, Jean-Marc~Schlenker, {\em On the infinitesimal rigidity of weakly convex polyhedra}, European J. Combin. {\bf 59}, 1080--1090 (2017)
	
	\bibitem[Du]{steve_dutch_enumeration} Steve~Dutch, {\em Enumeration of Polyhedra}, \url{http://www.uwgb.edu/dutchs/symmetry/polynum0.htm}
	
	
	\bibitem[Fl]{fleischner_transitive_1979} Herbert~Fleischner, Wilfried~Imrich, {\em Transitive planar graphs}, Math. Slovaca {\bf 29}, 97--106 (1979)
	
	\bibitem[Gr]{grunbaum_convex_2003} Branko~Grünbaum, {\em  Convex polytopes 2nd Ed.}, Springer-Verlag, (2003).
	
	
	\bibitem[Jo]{johnson_convex_1966} Norman~W.~Johnson, {\em Convex polyhedra with regular faces} Canad. J. Math. {\bf 18}, 169--200 (1966)
	
	\bibitem[Mo]{montcouquiol_deformations_2012}  Grégoire~Montcouquiol, {\em Deformations of hyperbolic convex polyhedra and cone-3-manifolds}, Geom. Dedi. {\bf 166}(1), 163--183 (2013), 
	
	
	\bibitem[St]{stoker_geometrical_1968} James~J.~Stoker, {\em Geometrical problems concerning polyhedra in the large}, Comm. Pure Appl. Math. {\bf 21}, 119--168 (1968)  
	
	
	\bibitem[Zi]{ziegler_lectures_1995} Günter~M.~Ziegler, {\em Lectures on Polytopes }, Springer-Verlag GTM {\bf 152}, (1995)
\end{thebibliography}


\end{document}